\documentclass[twoside,12pt,paperA4]{article}
\usepackage{amsmath,amssymb,graphicx,amsthm}
\usepackage[latin1]{inputenc}
\usepackage{anysize}
\marginsize{2cm}{2cm}{2cm}{3cm}

\usepackage{hyperref}
\hypersetup{colorlinks=true, citecolor=blue} 

\usepackage{tikz}
\usepackage{colortbl}
\usepackage{stmaryrd} 

\theoremstyle{definition}
  \newtheorem{definition}{Definition}[section]

\theoremstyle{plain}
  \newtheorem{theorem}{Theorem}[section]
  \newtheorem{proposition}{Proposition}[section]
  
  \newtheorem{lemma}{Lemma}[section]

\newcommand{\re}{\mathbb{R}}

\usepackage{dsfont}
\usepackage{multirow}

\usepackage{caption}
\usepackage{subcaption}

\usepackage{lineno} 
\title{Vanishing viscosity limit for Riemann solutions to a $2 \times 2$ hyperbolic system 
with linear damping}
\author{Richard De la cruz \thanks{Universidad Pedagógica y Tecnológica de Colombia, School of Mathematics and Statistics, 150003, Colombia. {\em e-mail:} \href{mailto:richard.delacruz@uptc.edu.co}{richard.delacruz@uptc.edu.co} } 
\\ Juan Juajibioy \thanks{Universidad Pedagógica y Tecnológica de Colombia, School of Mathematics and Statistics, 150003, Colombia. {\em e-mail:} \href{mailto:juan.juajibioy@uptc.edu.co}{juan.juajibioy@uptc.edu.co}}
}
\date{September 15, 2020}

\begin{document}
\maketitle

\begin{abstract}
In this paper, we propose a time-dependent viscous system and by using the vanishing viscosity method we show the existence of 
solutions for the Riemann problem
to a particular $2 \times 2$ system of conservation laws with linear damping.
\end{abstract}
\section{Introduction}
In this paper, we study the existence of solutions to the Riemann problem for the following hyperbolic system of conservation laws with linear damping
\begin{equation}
\label{system_ld}
  \begin{cases}
   u_t+\frac{1}{k+1}(u^{k+1})_x=-\alpha u,\\
   v_t+\left(vu^k \right)_x=0,
  \end{cases}
 \end{equation}
where $\alpha>0$ is a constant,
and the sign of $v$ is assumed to be unchanging. Thus for convenience, we assume
$v \ge 0$ throughout  this  paper.
The initial data is given by
\begin{equation} \label{datoRiemann}
 (v(x,0),u(x,0))=\begin{cases}
                (v_-,u_-), &\text{if } x<0,\\
                (v_+,u_+), &\text{if } x>0,
               \end{cases}
\end{equation}
for arbitrary constant states $(v_\pm,u_\pm)$ with $v_\pm>0$.
It is well known that the system \eqref{system_ld} is not strictly hyperbolic with eigenvalue $\lambda = u^k$ and right eigenvector ${\bf r}=(1,0)$. Moreover, $\nabla \lambda \cdot {\bf r}=0$ and therefore the system is linearly degenerate.
When $k=1$, the homogeneous case of the system \eqref{system_ld} is used to model the evolution of density inhomogeneities in matter in the universe \cite[B. Late nonlinear stage, 3. Sticky dust]{SZ}. 
The system \eqref{system_ld} belongs to the class of triangular systems. 
The triangular systems of conservation laws arises in a wide variety of models in physics and engineering, see for example \cite{Isaacson, SS} and the references therein.
 For this reason, the triangular systems have been studied by many authors and several rigorous results have been obtained for this.\\

In 1993, Joseph \cite{Joseph} considered the Riemann problem for the homogeneous case of the system \eqref{system_ld} with $k=1$. He used a parabolic regularization system to obtained an explicit formulae of the Riemann solutions. So, he constructed the weak limit of the approximation solution and this is defined as a delta shock wave type solution.
Recently, De la cruz \cite{Delacruz} solved the Riemann problem to the system \eqref{system_ld} when $k=1$. His work include classical Riemann solution and delta shock wave solution.\\

In this paper, we are interested in finding solutions to the Riemann problem for the system \eqref{system_ld} with inital data \eqref{datoRiemann}.   Therefore, we 
propose the following time-dependent viscous system
 \begin{equation}
\label{system_ld_eps}
  \begin{cases}
     u_t+\frac{1}{k+1}(u^{k+1})_x=\varepsilon \frac{1}{\alpha k} e^{-\alpha kt}(1-e^{-\alpha kt}) u_{xx}-\alpha u,\\
   v_t+\left(v u^k \right)_x=0,
  \end{cases}
 \end{equation}
(where $\alpha>0$ is a constant) with initial data \eqref{datoRiemann}. 
Observe that when $\alpha \to 0+$, we have that~$\lim\limits_{\alpha \to 0+} \frac{1}{\alpha k} e^{-\alpha kt}(1-e^{-\alpha kt})=t$.
The viscous system \eqref{system_ld_eps} is well motivated by scalar conservation law with time-dependent viscosity
\begin{equation*} \label{GBurgers}
u_t+ F(u)_x = G(t)u_{xx}.
\end{equation*}
where $G(t) > 0$ for $t > 0$.
When $F(u)=u^2$ the scalar equation is called the Burgers equation with time-dependent viscosity. The Burgers equation with time-dependent viscosity
was studied as a mathematical model of the propagation of the finite-amplitude sound waves in variable-area ducts, where $u$ is an acoustic variable, with the linear effects of changes in the duct area taken out, and the time-dependent viscosity $G(t)$ is  the  duct  area \cite{Crighton, DE, WZ}.
The reader can find results concerning to the existence, uniqueness and explicit solutions to the Burgers equation with time-dependent viscosity with suitable conditions for $G(t)$ in \cite{Crighton, CS, DE, Scott, WZ2, WZ, HZhang, ZW} and references cited therein.
The Burgers equation with time-dependent viscosity and linear damping
was studied in \cite{MVSK} and their results include explicit solutions for differents $G(t)$.\\
When $G(t)=\varepsilon t$ and $\varepsilon>0$, for systems of hyperbolic conservation laws with time-dependent viscosity we refered the works developed by Tupciev in \cite{Tupciev} and Dafermos in \cite{Dafermos}. The results obtained in \cite{Dafermos} and \cite{Tupciev} not including the delta shock waves solutions.
For systems of hyperbolic conservation laws with delta shock solutions
the reader may consult \cite{DS, Ercole, TZZ, HYang, YZ}.\\
When $G(t)$ is nonlinear, for systems of balance laws we refered the work \cite{Delacruz}.\\

Note that our proposal of the time-dependent viscous system \eqref{system_ld_eps} is a special case of the general systems of conservation laws with time-dependent viscous system. 
 Observe that if $(\widehat{v},\widehat{u})$ solves
\begin{equation} \label{sysVis1}
\begin{cases}
\widehat{u}_t+\frac{1}{k+1}e^{-\alpha kt}(\widehat{u}^{k+1})_x=\varepsilon \frac{1}{\alpha k}e^{-\alpha kt}(1-e^{-\alpha kt})\widehat{u}_{xx},\\
\widehat{v}_t+e^{-\alpha kt}(\widehat{v} \widehat{u}^k)_x=0,
\end{cases}
\end{equation} 
with initial condition
\begin{equation} \label{datsysVisc1}
 (\widehat{v}(x,0),\widehat{u}(x,0))=\begin{cases}
                (v_-,u_-), &\text{if } x<0,\\
                (v_+,u_+), &\text{if } x>0,
               \end{cases}
\end{equation}
then $(v,u)$ defined by $(v,u)=(\widehat{v},\widehat{u}e^{-\alpha t})$ solves the problem \eqref{system_ld_eps}--\eqref{datoRiemann}. 
We denote $(\widehat{v}^\varepsilon,\widehat{u}^\varepsilon)$ as $(\widehat{v},\widehat{u})$ when there is no confusion.
In order to solve the problem \eqref{sysVis1}--\eqref{datsysVisc1}, we introduce the similarity variable $\xi$ and solutions to \eqref{sysVis1} should approach for large times a similarity solution $(\widehat{v},\widehat{u})$ to \eqref{sysVis1} of the form $\widehat{v}(x,t)=\widehat{v}(\xi)$, $\widehat{u}(x,t)=\widehat{u}(\xi)$ and $\xi= a(t) x$ for some suitable smooth function $a(t) \ge 0$ for $t>0$ (more details on the similarity methods can be found in \cite{Barenblatt, Henriksen, Luo, Polyanin, Sachdev, Suto} and references therein).
Therefore, we introduce the similarity variable $\xi=\frac{\alpha kx}{1-e^{-\alpha kt}}$ and the system \eqref{sysVis1} can be written as follows
\begin{equation} \label{sysVis2}
\begin{cases} 
- \xi  \widehat{u}_\xi+\frac{1}{k+1}(\widehat{u}^{k+1})_\xi = \varepsilon \widehat{u}_{\xi \xi},\\
-\xi \widehat{v}_\xi+(\widehat{v}\widehat{u}^k)_\xi=0,\\
\end{cases}
\end{equation} 
and the initial data \eqref{datsysVisc1} changes to the boundary condition
\begin{equation} \label{datsysVisc2}
(\widehat{v}(\pm \infty), \widehat{u}(\pm \infty))=(v_\pm,u_\pm).
\end{equation}
Note that when $\alpha \to 0+$, the similarity variable $\xi$ converges to $x/t$ which is well used in many methods to study the behavior and structure of solutions of nonlinear hyperbolic systems of conservation laws.
Notice that when $\varepsilon \to 0+$, the system \eqref{sysVis1} becomes
\begin{equation} \label{sys1}
\begin{cases}
\widehat{u}_t+\frac{1}{k+1}e^{-\alpha kt}(\widehat{u}^{k+1})_x=0,\\
\widehat{v}_t+ e^{-\alpha kt} \left( \widehat{v} \widehat{u}^{k} \right)_x  =  0.
\end{cases}
\end{equation} 
Using the vanishing viscosity method, and following works by Tan, Zhang and Zheng \cite{TZZ}  
and Ercole \cite{Ercole} with some appropriate modifications, we show the existence of solutions for system \eqref{sysVis2} with boundary condition \eqref{datsysVisc2}.
After, we study the behavior of the solutions $(\widehat{v}^\varepsilon, \widehat{u}^\varepsilon)$ as $\varepsilon \to 0+$ to obtain classical Riemann solution and
delta shock wave solution for the system \eqref{sys1}. 
Finally, as $(v(x,t),u(x,t))=(\widehat{v}(x,t), \widehat{u}(x,t) e^{-\alpha t})$, the solutions of \eqref{sys1} are used to obtain solutions of the original system \eqref{system_ld}. \\

The outline of the remaining of the paper is as follows. In Section 2, we show the existence of solutions to the viscous system \eqref{sysVis2} with boundary condition \eqref{datsysVisc2}. In Section 3, we study the behavior of the solutions $(\widehat{v}^\varepsilon, \widehat{u}^\varepsilon)$ as $\varepsilon \to 0+$ and we solve the Riemann problem to the system \eqref{sysVis1} without viscosity. In Section 4, we show classical Riemann solution and delta shock solution for the nonhomogeneous system \eqref{system_ld}. Final remarks are given in Section 5.

 \section{Existence of solutions to the viscous system \texorpdfstring{\eqref{sysVis2}-\eqref{datsysVisc2}}{(7)-(8)} }
 
 Considering the first equation in \eqref{sysVis2} with boundary conditions, we have
 \begin{equation} \label{ViscEq1}
 \begin{cases}
 -\xi \widehat{u}_\xi +\frac{1}{k+1}(\widehat{u}^{k+1} )_\xi = \varepsilon \widehat{u}_{\xi \xi},\\
 \widehat{u}(\pm \infty) = u_\pm.
 \end{cases}
 \end{equation}
Now, based on the ideas of Dafermos \cite{Dafermos}, we consider the following boundary value problem with parameters $\mu \in [0,1]$ and $R > 1$,
 \begin{equation} \label{ViscEq2}
 \begin{cases}
 -\xi \widehat{u}_\xi +\frac{\mu}{k+1} (\widehat{u}^{k+1} )_\xi = \varepsilon \widehat{u}_{\xi \xi},\\
 \widehat{u}(\pm R) = \mu u_\pm.
 \end{cases}
 \end{equation}
\begin{lemma} \label{Lemma2.1}
Let $\widehat{u}(\xi)$ be a solution of \eqref{ViscEq2} on $[-R,R]$ for some $\mu >0$. Suppose that $u_- \neq u_+$. Then, $\widehat{u}$  is a strictly 
monotonic function on $[-R, R]$.
\end{lemma}
\begin{proof}
Observe that from \eqref{ViscEq2} we have that 
\begin{equation} \label{Ig1}
\widehat{u}'(\xi) = \widehat{u}'(\zeta) \exp \left( \int_\zeta^\xi \frac{\mu\widehat{u}^k(s)-s}{\varepsilon} ds \right)
\end{equation}
for any $\zeta \in [-R,R]$. 
Suppose $\xi_1 \in [-R, R]$ is a critical point of $\widehat{u}(\xi)$, which implies $\widehat{u}'(\xi_1)=0$. Then, from \eqref{Ig1} we have that $\widehat{u}'(\xi)=0$ for all $\xi \in [-R,R]$, and therefore $\widehat{u}(\xi)$ is constant on $[-R, R]$. But, this contradicts the fact that $u_- \neq u_+$. Thus, $\widehat{u}(\xi)$ is monotone. The monotonicity of $\widehat{u}(\xi)$ depends on the value of $\widehat{u}'(\xi_0)$. 
If $u_- >u_+$, then $\widehat{u}(\xi)$ is strictly decreasing on $[-R,R]$.
When $u_- <u_+$, we have that $\widehat{u}(\xi)$ is strictly increasing on $[-R,R]$.
\end{proof} 
 
\begin{theorem}
Suppose that $u_->u_+$.
For every $\varepsilon >0$, there exists a smooth solution (not necessarily unique) of
 \eqref{ViscEq1}.
\end{theorem}
\begin{proof}
From Lemma \ref{Lemma2.1}, we have $\sup\limits_{-R<\xi<R} |u(\xi)| \le \max \{u_-,u_+ \}$ which does not depend on $\mu$ and $R$. Now, from Theorem 3.1 in \cite{Dafermos} we conclude that there exists a solution of \eqref{ViscEq1}. Moreover, if $u_- >u_+$, then the solution is decreasing on $(-\infty, \infty)$.
For $u_- <u_+$, the solution is increasing on $(-\infty, \infty)$.
\end{proof}

\begin{proposition} \label{Prop2.1}
Let $w(\xi)$ be a smooth solution of \eqref{ViscEq1}. Then,
$$ \left| w'(\xi) \right| \le \left| w'(0) \right| \exp \left( \frac{2u_-^k |\xi|-\xi^2}{2 \varepsilon} \right), \qquad -\infty <\xi <+\infty.$$
\end{proposition}
\begin{proof}
Multiplying the equation of \eqref{ViscEq1} by $\exp( \frac{\xi^2}{2 \varepsilon} )$, we have
$$ \frac{d}{d\xi} \left( w'(\xi) \exp( \frac{\xi^2}{2 \varepsilon} ) \right) = \frac{1}{\varepsilon} w^k(\xi) w'(\xi)\exp( \frac{\xi^2}{2 \varepsilon} ) $$
which yields the estimate
$$ \left| w'(\xi) \right| \le \left| w'(0) \right| \exp \left( \frac{2u_-^k |\xi|-\xi^2}{2 \varepsilon} \right).$$
\end{proof}

\begin{theorem}
Let $u_1(\xi)$ and $u_2(\xi)$ be two smooth solutions of \eqref{ViscEq1}. Then, $u_1=u_2$.
\end{theorem}
\begin{proof}
Let $\widehat{u}_1(\xi)$ and $\widehat{u}_2(\xi)$ be solutions of the problem \eqref{ViscEq1} and $\widehat{U}(\xi):=\widehat{u}_1(\xi)-\widehat{u}_2(\xi)$. Then, from \eqref{ViscEq1} we have that $\widehat{U}$ is a smooth solution of the boundary value problem
 \begin{equation} \label{ViscEq3}
 \begin{cases}
 -\xi \widehat{U}_\xi +\mu ( \widehat{U}h )_\xi = \varepsilon \widehat{U}_{\xi \xi},\\
 \widehat{U}(\pm \infty) = 0,
 \end{cases}
 \end{equation}
 where $h(\theta)=\int_0^1 (k+1) (\widehat{u}_1(\xi) + (\widehat{u}_2(\xi)-\widehat{u}_1(\xi)) \theta )^k d\theta$. We note that $h(\xi)$ is bounded.
 Observe that from Proposition \ref{Prop2.1}, we have
 $$ |\widehat{U}'(\xi)| \le (|u_1'(0)|+|u_2'(0)|) \exp \left( \frac{2 u_-^k |\xi|-\xi^2}{2 \varepsilon} \right) $$
 and $\widehat{U}'(\xi)$ decays rapidly to zero when $|\xi| \to \infty$ for each fixed $\varepsilon > 0$. Therefore, when $\lim\limits_{\xi \to \pm \infty} \widehat{U}(\xi)=0$ we have $\lim\limits_{\xi \to \pm \infty} \xi \widehat{U}(\xi) =0$. \\ 
 Let us suppose that $\widehat{U}$ is not the null function. Let $a$ and $b$ be consecutive zeros of $\widehat{U}$ with $-\infty \le a<b \le +\infty$.
 So, integrating \eqref{ViscEq3} by parts on $(a, b)$ we find
\begin{equation} \label{eqInt}
\varepsilon (\widehat{U}'(b)-\widehat{U}'(a)) = \int_a^b \widehat{U}(\xi) d\xi.
\end{equation}
 Now, if $\widehat{U}>0$ on $(a,b)$, then $\widehat{U}'(b) \le 0 \le \widehat{U}'(a)$ and $\int_a^b \widehat{U}(\xi) d\xi >0$. But, we have a contradiction with \eqref{eqInt} because in this case \eqref{eqInt} implies $\widehat{U}'(b) > \widehat{U}'(a)$. In similar way,
if $\widehat{U}<0$ on $(a,b)$, then $\widehat{U}'(b) \ge 0 \ge \widehat{U}'(a)$ and $\int_a^b \widehat{U}(\xi) d\xi >0$, which again contradicts with \eqref{eqInt}. Thus, we conclude that $\widehat{U} \equiv 0$.
\end{proof}
Putting $\widehat{u}(\xi)$ into the second equation of \eqref{sysVis2} with boundary conditions \eqref{datsysVisc2},  we get
\begin{equation} \label{Eqv2}
\begin{cases}
-\xi \widehat{v}_\xi + (\widehat{v} \widehat{u}^k )_\xi =0,\\
\widehat{v}(\pm \infty) =v_\pm.
\end{cases}
\end{equation}
The singularity point of \eqref{Eqv2} is given by the unique solution of $(\widehat{u}(\xi))^k = \xi$ and it is denoted by $\xi_\sigma^\varepsilon$.
Observe that the solution of \eqref{Eqv2} can be obtained by pasting together the two solutions in the regions $(-\infty,\xi_\sigma^\varepsilon)$ and $(\xi_\sigma^\varepsilon,+\infty)$.
Now integrating \eqref{Eqv2} from $-\infty$ to $\xi$ for $\xi < \xi_\sigma^\varepsilon$, we obtain
\begin{equation} \label{solRho1}
\widehat{v}_1(\xi)=v_- \exp \left( -\int_{-\infty}^\xi \frac{(\widehat{u}^k(s))'}{\widehat{u}^k(s)-s} ds \right).
\end{equation}
On the other hand, integrating \eqref{Eqv2} from $\xi$ to $+\infty$ for $\xi > \xi_\sigma^\varepsilon$, we obtain
\begin{equation} \label{solRho2}
\widehat{v}_2(\xi)=v_+ \exp \left( \int_\xi^{+\infty} \frac{(\widehat{u}^k(s))'}{\widehat{u}^k(s)-s} ds \right).
\end{equation}
\begin{lemma}
Suppose $u_->u_+$. Let 
 \begin{equation} \label{solweak1}
  \widehat{v}(\xi)=\begin{cases}
             \widehat{v}_1(\xi), &\mbox{if } \xi < \xi_\sigma^\varepsilon,\\
          \widehat{v}_2(\xi), &\mbox{if } \xi > \xi_\sigma^\varepsilon,
            \end{cases}
 \end{equation}
where $\xi_\sigma^\varepsilon$ is the unique solution of the equation $(\widehat{u}(\xi))^{k}=\xi$  (which solution exists because $u_->u_+$ and $\widehat{u}$ is decreasing), $\widehat{v}_1$ and $\widehat{v}_2$ are defined by \eqref{solRho1} and \eqref{solRho2}, respectively.
 Then $\widehat{v} \in L^1(-\infty,+\infty)$, $\widehat{v}$ is continuous in $(-\infty,\xi_\sigma^\varepsilon) \cup (\xi_\sigma^\varepsilon,+\infty)$ and it is a weak solution for 
 \begin{equation} \label{eq1}
	   -\xi \widehat{v}_\xi+(\widehat{v} \widehat{u}^k)_\xi=0.
	\end{equation} 
\end{lemma} 
 
\begin{proof}
Note that from the formula \eqref{solRho1}, $\widehat{v}_1(\xi)$ is monotonically increasing (or decreasing) when $v_-<0$ (or $v_->0$) in the interval $(-\infty,\xi_\sigma^\varepsilon)$, and from  \eqref{solRho2} that $\widehat{v}_2(\xi)$ is monotonically decreasing (or increasing) when $v_-<0$ (or $v_->0$) in the interval $(\xi_\sigma^\varepsilon, +\infty)$. Also, we have 
$$ \lim_{\xi \to \xi_\sigma^\varepsilon-} \widehat{v}_1(\xi) = \pm \infty, \quad
\lim_{\xi \to \xi_\sigma^\varepsilon+} \widehat{v}_2(\xi) = \pm \infty. $$

 The equation \eqref{eq1} can be rewritten as
 \begin{equation} \label{eq4}
  ((\widehat{u}(\xi))^k-\xi)\widehat{v}'+\widehat{v}((\widehat{u}(\xi))^k)'=0.
 \end{equation}
 Now, we can show that $\widehat{v} \in L^1[\xi_1,\xi_2]$ for any interval $[\xi_1,\xi_2]$ containing $\xi_\sigma^\varepsilon$. In fact,
integrating \eqref{eq4} on $[\xi_1,\xi]$ for $\xi_1<\xi<\xi_\sigma^\varepsilon$ , we get
 \begin{equation} \label{eq5}
  ((\widehat{u}(\xi))^k-\xi)\widehat{v}_1(\xi)-((\widehat{u}(\xi_1))^k-\xi_1)\widehat{v}_1(\xi_1)+  \int_{\xi_1}^\xi \widehat{v}_1(s)ds=0.
 \end{equation}
 Let
 $$ p(\xi)=\int_{\xi_1}^\xi \widehat{v}_1(s)ds, \quad A_1=((\widehat{u}(\xi_1))^k-\xi_1)\widehat{v}_1(\xi_1) \quad \mbox{ and } a(\xi)=((\widehat{u}(\xi))^k-\xi). $$
 Then \eqref{eq5} can be written as
 \begin{equation*}
  \begin{cases}
   a(\xi)p'(\xi)+p(\xi)=A_1,\\
   p(\xi_1)=0.
  \end{cases}
 \end{equation*}
It follows that
$$ p(\xi)=A_1 \left\{ 1-\exp \left( -\int_{\xi_1}^\xi \frac{ds}{a(s)} \right) \right\}. $$
Noting that $a(\xi)>0$ and $a(\xi)=O(|\xi-\xi_\sigma|)$ as $\xi \to \xi_\sigma^\varepsilon-$, we obtain
\begin{equation} \label{eqA}
 \lim_{\xi \to \xi_\sigma^\varepsilon-} \int_{\xi_1}^\xi \widehat{v}_1(s) ds = \lim_{\xi \to \xi_\sigma^\varepsilon-} p(\xi)=A_1. 
\end{equation}
Hence
\begin{equation} \label{eqEst1}
 \lim_{\xi \to \xi_\sigma^\varepsilon-} ((\widehat{u}(\xi))^k-\xi)\widehat{v}_1(\xi)=0.
\end{equation}
Similarly, one can get
\begin{align}
 \lim_{\xi \to \xi_\sigma^\varepsilon+} \int_{\xi}^{\xi_2} \widehat{v}_2(s) ds =A_2, \label{eqB}\\
 \lim_{\xi \to \xi_\sigma^\varepsilon+} ((\widehat{u}(\xi))^k-\xi)\widehat{v}_2(\xi)=0, \nonumber
\end{align}
where $A_2=((\widehat{u}(\xi_2))^k-\xi_2)\widehat{v}_2(\xi_2)$. The equalities \eqref{eqA} and \eqref{eqB} imply that $\widehat{v}(\xi) \in L^1([\xi_1,\xi_2])$.\\
Given an arbitrary function $\phi \in C_0^\infty ([\xi_1,\xi_2])$, we can show that
\begin{equation*}
 I \equiv -\int_{\xi_1}^{\xi_2} ((\widehat{u}(\xi))^k-\xi)\widehat{v}(\xi) \phi'(\xi) d\xi +  \int_{\xi_1}^{\xi_2} \widehat{v}(\xi)\phi(\xi) d\xi=0.
\end{equation*}
Indeed, for any $\widetilde{\xi}_1, \widetilde{\xi}_2$ such that $\xi_1<\widetilde{\xi}_1 <\xi_\sigma^\varepsilon <\xi_2 <R$  we can write $I=I_1+I_2+I_3$, where 
\begin{align*}
 I_1&= \int_{\xi_1}^{\widetilde{\xi_1}} (-((\widehat{u}(\xi))^k-\xi)\widehat{v}(\xi) \phi'(\xi)+\widehat{v}(\xi)\phi(\xi)) d\xi,\\
 I_2&= \int_{\widetilde{\xi_1}}^{\widetilde{\xi_2}} (-((\widehat{u}(\xi))^k-\xi)\widehat{v}(\xi) \phi'(\xi)+\widehat{v}(\xi)\phi(\xi)) d\xi \mbox{ and }\\
 I_3&= \int_{\widetilde{\xi_2}}^{\xi_2} (-((\widehat{u}(\xi))^k-\xi)\widehat{v}(\xi) \phi'(\xi)+\widehat{v}(\xi)\phi(\xi)) d\xi.
\end{align*}
Observe that 
\begin{align*}
 |I_1|&=\left| -((\widehat{u}(\widetilde{\xi_1}))^k-\widetilde{\xi}_1)\widehat{v}_1(\widetilde{\xi}_1)\phi(\widetilde{\xi}_1)+\int_{\xi_1}^{\widetilde{\xi_1}} ((((\widehat{u}(\xi))^k-\xi)\widehat{v}(\xi))' \phi(\xi)+\widehat{v}(\xi)\phi(\xi))) d\xi \right|\\
 &=\left| ((\widehat{u}(\widetilde{\xi}_1))^k-\widetilde{\xi}_1)\widehat{v}_1(\widetilde{\xi}_1)\phi(\widetilde{\xi}_1) \right|.
\end{align*}
By \eqref{eqEst1}, we have that $$\lim_{\widetilde{\xi}_1 \to \xi_\sigma^\varepsilon-} |I_1|=\lim_{\widetilde{\xi}_1 \to \xi_\sigma^\varepsilon-} \left| ((\widehat{u}(\widetilde{\xi}_1))^k-\widetilde{\xi}_1)\widehat{v}_1(\widetilde{\xi}_1)\phi(\widetilde{\xi}_1) \right|=0.$$
In similar way, we show that 
$$\lim_{\widetilde{\xi}_2 \to \xi_\sigma^\varepsilon+} |I_3|=\lim_{\widetilde{\xi}_2 \to \xi_\sigma^\varepsilon+} \left| ((\widehat{u}(\widetilde{\xi}_2))^k-\widetilde{\xi}_2)\widehat{v}_2(\widetilde{\xi}_2)\phi(\widetilde{\xi}_2) \right|=0.$$
Since $\widehat{v} \in L^1([\xi_1,\xi_2])$, 
$$ |I_2| \le \int_{\widetilde{\xi_1}}^{\widetilde{\xi_2}} |-((\widehat{u}(\xi))^k-\xi)\phi'(\xi)+\phi(\xi)||\widehat{v}(\xi)|d\xi \to 0, \quad \mbox{ as } \widetilde{\xi}_1\to \xi_\sigma^\varepsilon-, \widetilde{\xi}_2\to \xi_\sigma^\varepsilon+. $$
But $I$ is independent of $\widetilde{\xi}_1$ and $\widetilde{\xi}_2$, so $I=0$. Therefore, $\widehat{v}$ defined in \eqref{solweak1} is a weak solution.
\end{proof}

\begin{lemma} \label{Lemma2.3}
Suppose $u_-<u_+$. Let 
 \begin{equation*} 
  \widehat{v}(\xi)=\begin{cases}
             \widehat{v}_1(\xi), &\mbox{if } \xi < \xi_{\sigma_1}^\varepsilon,\\
             0, &\mbox{if }  \xi_{\sigma_1}^\varepsilon \le \xi \le \xi_{\sigma_2}^\varepsilon,\\
          \widehat{v}_2(\xi), &\mbox{if } \xi > \xi_{\sigma_2}^\varepsilon,
            \end{cases}
 \end{equation*}
where $\widehat{v}_1$ and $\widehat{v}_2$ are defined by \eqref{solRho1} and \eqref{solRho2}, respectively, $\xi_{\sigma_1}^\varepsilon \le \xi_{\sigma_2}^\varepsilon$
satisfying $\xi_{\sigma_1}^\varepsilon = \min \{ \xi : (\widehat{u}(\xi))^k=\xi \}$,
$\xi_{\sigma_2}^\varepsilon = \max \{ \xi : (\widehat{u}(\xi))^k=\xi \}$ and 
$\lim\limits_{\xi \to \xi_{\sigma_1}^\varepsilon-} \widehat{v}_1(\xi)=\lim\limits_{\xi \to \xi_{\sigma_2}^\varepsilon+} \widehat{v}_2(\xi)=0$.
 Then $\widehat{v} \in L^1(-\infty,+\infty)$, $\widehat{v}_1$ is decreasing in
 $(-\infty,\xi_{\sigma_1}^\varepsilon)$,  $\widehat{v}_2$ is increasing in
 $(\xi_{\sigma_2}^\varepsilon,+\infty)$, $\widehat{v}$ is continuous on the intervals $(-\infty,\xi_{\sigma_1}^\varepsilon)$ and $(\xi_{\sigma_2}^\varepsilon,+\infty)$, and it is a weak solution for $ -\xi \widehat{v}_\xi+(\widehat{v} \widehat{u}^k)_\xi=0$.
\end{lemma} 
\begin{proof}
Observe that $u_-<u_+$ implies $\widehat{u}$ is  increasing.
Consider now the function $s \mapsto s- \widehat{u}^k(s)$ which is continuous and approaches $\pm \infty$ as $s \to \pm \infty$. Hence, there exist finite quantities $\xi_{\sigma_1}^\varepsilon = \min \{ \xi : (\widehat{u}(\xi))^k=\xi \}$
and $\xi_{\sigma_2}^\varepsilon = \max \{ \xi : (\widehat{u}(\xi))^k=\xi \}$.
One has $s-(\widehat{u}(s))^k <0$ on $(-\infty, \xi_{\sigma_1}^\varepsilon)$ and 
$s-(\widehat{u}(s))^k >0$ on $( \xi_{\sigma_2}^\varepsilon, +\infty)$.
Moreover, we can get $\xi_{\sigma_1}^\varepsilon \le \xi_{\sigma_2}^\varepsilon$.
We now claim that 
\begin{equation} \label{IntInf1}
\lim_{\xi \to \xi_{\sigma_2}^\varepsilon+}\int_{\xi}^{+\infty} \frac{(\widehat{u}^k(s))'}{s-\widehat{u}^k(s)}ds = +\infty.
\end{equation}
In fact, for $R$ fixed and $\xi_{\sigma_2}^\varepsilon < \xi < R$ we have
\begin{align*}
\int_{\xi}^R \frac{(\widehat{u}^k(s))'}{s-\widehat{u}^k(s)}ds &=
(\widehat{u}^k(\zeta))' \int_{\xi}^R \frac{ds}{s-\widehat{u}^k(s)} 
\ge (\widehat{u}^k(\zeta))' \int_{\xi}^R \frac{ds}{s-\widehat{u}^k(\xi)} \\
&= -(\widehat{u}^k(\zeta))' \ln \left( \frac{\xi-\widehat{u}^k(\xi)}{R-\widehat{u}^k(\xi)}    \right) \to +\infty, \quad \mbox{as } \xi \to \xi_{\sigma_2}^\varepsilon+,
\end{align*}
where $\xi \le \zeta \le R$. Now, from \eqref{solRho2} and \eqref{IntInf1} we get
$$ \lim_{\xi \to \xi_{\sigma_2}^\varepsilon+} \widehat{v}_2(\xi) =0. $$
In a similar way, we can obtain $\lim\limits_{\xi \to \xi_{\sigma_1}^\varepsilon-} \widehat{v}_1(\xi) =0$.
The monotonicity of $\widetilde{v}_1$ and $\widehat{v}_2$ is obvious.
 When $\xi_{\sigma_1}^\varepsilon \le \xi \le \xi_{\sigma_2}^\varepsilon$, from \eqref{eq4} we have
$$ \int_{\xi_{\sigma_1}^\varepsilon}^{\xi_{\sigma_2}^\varepsilon} ((\widehat{u}(\xi))^k \widehat{v}'-\xi \widehat{v}'+\widehat{v}((\widehat{u}(\xi))^k)' d\xi =0$$
or 
$$ \left. (((\widehat{u}(\xi))^k -\xi) \widehat{v}(\xi) \right|_{\xi_{\sigma_1}^\varepsilon}^{\xi_{\sigma_2}^\varepsilon}
+ \int_{\xi_{\sigma_1}^\varepsilon}^{\xi_{\sigma_2}^\varepsilon} \widehat{v}(\xi) d\xi =0$$
which implies that $\widehat{v}(\xi)=0$.
\end{proof}
\section{The limit solutions of \texorpdfstring{\eqref{sysVis1}--\eqref{datsysVisc1}}{(5)--(6)} as viscosity vanishes}
In this section, we are interested in analyzing the behavior of the solutions $(\widehat{v}^\varepsilon,\widehat{u}^\varepsilon)$ of \eqref{sysVis2}--\eqref{datsysVisc2} as $\varepsilon \to 0+$ to stablished the solutions of \eqref{sysVis1}--\eqref{datsysVisc1}.\\

{\bf Case 1. $u_->u_+$} 
\begin{lemma} \label{lemmaA}
 Let $\xi_\sigma^\varepsilon$ be the unique point satisfying $(\widehat{u}^\varepsilon(\xi_\sigma^\varepsilon))^k=\xi_\sigma^\varepsilon$, and let $\xi_\sigma$ be the limit $ \xi_\sigma=\lim\limits_{\varepsilon \to 0+} \xi_\sigma^\varepsilon $
 (passing to a subsequence if necessary). Then for any $\eta >0$,
 \begin{align*}
  &\lim_{\varepsilon \to 0+} \widehat{u}_\xi^\varepsilon(\xi)=0, \qquad \mbox{for } |\xi-\xi_\sigma| \ge \eta,\\
  &\lim_{\varepsilon \to 0+} \widehat{u}^\varepsilon(\xi)=\begin{cases}
                                                 u_-, &\mbox{if } \xi \le \xi_\sigma -\eta,\\
                                                 u_+, &\mbox{if } \xi \ge \xi_\sigma +\eta,\\
                                                \end{cases}
 \end{align*}
uniformly in the above intervals.
Moreover, $\xi_\sigma = \frac{1}{k+1}\sum\limits_{j=0}^k u_-^{k-j}u_+^j$
and $\xi_\sigma[u]-\frac{1}{k+1}[u^{k+1}]=0$.
\end{lemma}
\begin{proof} To simplify the notation in this proof, we shall use $\widehat{v}$, $\widehat{u}$ instead of $\widehat{v}^\varepsilon$, $\widehat{u}^\varepsilon$.

 Take $\xi_3=\xi_\sigma-\eta/2$, and let $\varepsilon$ be so small such that $\xi_\sigma^\varepsilon >\xi_3+\eta/4$. 

Now, integrating the first equation of \eqref{sysVis2} twice on $[\xi,\xi_3]$, we get
 \begin{align*}
  \widehat{u}(\xi_3)-\widehat{u}(\xi)&=\widehat{u}'(\xi_3) \int_{\xi}^{\xi_3} \exp \left( -\int_r^{\xi_3} \frac{(\widehat{u}(s))^k-s}{\varepsilon}ds \right) dr
  \le \widehat{u}'(\xi_3) \int_\xi^{\xi_3} \exp \left( - \int_r^{\xi_3} \frac{u_-^k-s}{\varepsilon} ds \right) dr\\
  &= \widehat{u}'(\xi_3) \int_\xi^{\xi_3} \exp \left( \frac{1}{\varepsilon} \left( \left(u_-^k-\xi_3 \right)(r-\xi_3)-\frac12(r-\xi_3)^2 \right) \right) dr\\
  &= \widehat{u}'(\xi_3) \int_{\xi-\xi_3}^0 \exp \left( \frac{1}{\varepsilon} \left( \left(u_-^k-\xi_3 \right)r-\frac12 r^2 \right) \right) dr.
 \end{align*}
Letting $\xi \to -\infty$, we get
\begin{align*}
 u_+-u_- &\le \widehat{u}'(\xi_3) \int_{-\infty}^0 \exp \left( \frac{1}{\varepsilon} \left( \left(u_-^k-\xi_3 \right)r-\frac12 r^2 \right) \right) dr\\
 &\le \widehat{u}'(\xi_3) \int_{0}^{2\varepsilon} \exp \left( -\frac{1}{\varepsilon} \left( \left(u_-^k-\xi_3 \right)r+\frac12 r^2 \right) \right) dr\\
 &\le \widehat{u}'(\xi_3) \sqrt{\varepsilon}A_3
\end{align*}
for $0\le \varepsilon \le 1$, where $A_3$ is a constant independent of $\varepsilon$. Thus
\begin{equation*}
 |\widehat{u}'(\xi_3)| \le \frac{u_--u_+}{\sqrt{\varepsilon}A_3}.
\end{equation*}
So
\begin{equation} \label{eqC}
 |\widehat{u}'(\xi)| \le \frac{u_--u_+}{\sqrt{\varepsilon}A_3}\exp \left( -\int_\xi^{\xi_3} \frac{(\widehat{u}(s))^k-s}{\varepsilon}ds \right).
\end{equation}
Noticing that 
\begin{align*}
(\widehat{u}(s))^k-s &=((\widehat{u}(s))^k-(\widehat{u}(\xi_\sigma^\varepsilon))^k)-(s-\xi_\sigma^\varepsilon) 
= (k (\widehat{u}(\theta))^{k-1}u'(\theta)-1)(s-\xi_\sigma^\varepsilon) \ge \frac{\eta}{4}
\end{align*}
for $s \le \xi_3$ and from \eqref{eqC} we have
\begin{equation*}
 |\widehat{u}'(\xi)| \le \frac{u_--u_+}{\sqrt{\varepsilon}A_3}\exp \left( -\frac{\eta}{4\varepsilon} \left( \xi_3-\xi \right) \right)
\end{equation*}
which implies that
$$ \lim_{\varepsilon \to 0+} \widehat{u}_\xi^\varepsilon(\xi)=0, \qquad \mbox{uniformly for }\xi \le \xi_\sigma-\eta. $$
Now, we choose $\xi$ and $\xi_4$ such that $\xi<\xi_4 \le \xi_\sigma-\eta$.
From
$$ \widehat{u}(\xi_4)-\widehat{u}(\xi)=\widehat{u}'(\xi_4)\int_\xi^{\xi_4} \exp \left(- \int_r^{\xi_4}
\frac{(\widehat{u}(s))^k-s}{\varepsilon} ds \right) dr, $$
we get
\begin{align*}
 | \widehat{u}(\xi_4)-\widehat{u}(\xi)| &\le |\widehat{u}'(\xi_4)| \int_\xi^{\xi_4} \exp \left(
\frac{A_4}{\varepsilon}(r-\xi_4) \right)dr \le \frac{\varepsilon}{A_4}|\widehat{u}'(\xi_4)| \left(1-\exp \left(
\frac{A_4}{\varepsilon}(\xi-\xi_4) \right) \right),
\end{align*}
where $A_4=(\widehat{u}(\xi_4))^k - \xi_4$.
When $\xi \to-\infty$, we obtain
\begin{equation*}
 |\widehat{u}(\xi_4)-u_-| \le \frac{\varepsilon}{A_4}|\widehat{u}'(\xi_4)|,
\end{equation*}
which implies that
\begin{equation*}
 \lim_{\varepsilon \to 0+} \widehat{u}^\varepsilon(\xi)=u_-, \qquad \mbox{uniformly
for } \xi \le \xi_\sigma-\eta.
\end{equation*}
The results for $\xi \ge \xi_\sigma+\eta$ can be obtained analogously.\\
In fact, let $\phi \in C_0^\infty((\xi_1,\xi_2))$ where $\xi_1 <\xi_\sigma < \xi_2$, From \eqref{ViscEq1} we have
\begin{equation} \label{weAprx1}
\int_{\xi_1}^{\xi_2} \widehat{u}(\xi) \left( (\xi \phi(\xi))' -\frac{1}{k+1}\widehat{u}^{k}(\xi)\phi'(\xi) \right) d\xi =\varepsilon \int_{\xi_1}^{\xi_2} \widehat{u}(\xi)\phi''(\xi) d\xi.
\end{equation}
Passing limit $\varepsilon \to 0+$ in \eqref{weAprx1}, we get
\begin{equation*}
\int_{\xi_1}^{\xi_\sigma} u_- \left( (\xi \phi(\xi))' -\frac{1}{k+1}u_-^{k} \phi'(\xi) \right) d\xi+\int_{\xi_\sigma}^{\xi_2} u_+ \left( (\xi \phi(\xi))' -\frac{1}{k+1}u_+^{k}\phi'(\xi) \right) d\xi =0.
\end{equation*}
or 
\begin{equation*}
u_-  \xi_\sigma \phi(\xi_\sigma) -\frac{1}{k+1}u_-^{k+1} \phi(\xi_\sigma)- u_+ \xi_\sigma \phi(\xi_\sigma) +\frac{1}{k+1}u_+^{k+1}\phi(\xi_\sigma) =0
\end{equation*}
which yields $\xi_\sigma = \frac{1}{k+1}\sum\limits_{j=0}^k u_-^{k-j}u_+^j$ for arbitrary $\phi$.
\end{proof}
\begin{lemma} \label{lemmaB}
For any $\eta>0$,
 \begin{equation*}
  \lim_{\varepsilon \to 0+}\widehat{v}^\varepsilon(\xi)=\begin{cases}
                             v_-, &\mbox{if } \xi<\xi_\sigma-\eta,\\
                             v_+, &\mbox{if } \xi>\xi_\sigma+\eta,
                            \end{cases}
 \end{equation*}
uniformly, with respect to $\xi$.
\end{lemma}
\begin{proof}
Take $\varepsilon_0>0$ so small such that $|\xi_\sigma^\varepsilon-\xi_\sigma| < \frac{\eta}{2}$ whenever $0<\varepsilon<\varepsilon_0$. For any
$\xi \ge \xi_\sigma+\eta$ and $\varepsilon <\varepsilon_0$, we have $$ \xi >\xi_\sigma^\varepsilon+\frac{\eta}{2} $$
and
\begin{equation*} \label{Intv}
\widehat{v}^\varepsilon(\xi)=v_+ \exp \left( \int_{\xi}^\infty \frac{((\widehat{u}^\varepsilon(s))^k)'}{(\widehat{u}^\varepsilon(s))^k-s}ds \right).
\end{equation*}
For any $s\in [\xi,+\infty)$, we have
\begin{align*}
 (\widehat{u}^\varepsilon(s))^k-s &< (\widehat{u}^\varepsilon(\xi))^k-\xi=(1-((\widehat{u}^\varepsilon(\zeta))^k)')(\xi_\sigma^\varepsilon-\xi)
\le -\frac{\eta}{2}.
\end{align*}
As $\widehat{u}$ is decreasing, we have that $((\widehat{u}(\xi))^k)'=k(\widehat{u}(\xi))^{k-1}\widehat{u}'(\xi)<0$, and
$$ \frac{((\widehat{u}(s))^k)'}{(\widehat{u}^\varepsilon(s))^k-s} <-\frac{2}{\eta}((\widehat{u}(s))^k)' , \qquad \mbox{for any } s \in [\xi,+\infty). $$
Now, in the last inequality, integrating on $[\xi,+\infty)$ we have
$$ 0\le \int_{\xi}^\infty \frac{((\widehat{u}^\varepsilon(s))^k)'}{(\widehat{u}^\varepsilon(s))^k-s} ds \le -\frac{2}{\eta}\int_{\xi}^\infty((u^\varepsilon(s))^k)' ds= -\frac{2}{\eta}(u_+^k-(\widehat{u}^\varepsilon(\xi))^k), $$
so
\begin{equation} \label{inq1}
1 \le \exp \left( \int_{\xi}^\infty \frac{((\widehat{u}^\varepsilon(s))^k)'}{(\widehat{u}^\varepsilon(s))^k-s} ds \right) \le \exp \left( -\frac{2}{\eta}(u_+^k-(\widehat{u}^\varepsilon(\xi))^k) \right).
\end{equation}
By Lemma~\ref{lemmaA} we have that $ \lim\limits_{\varepsilon \to 0+} \widehat{u}^\varepsilon(\xi)=u_+$, and from \eqref{inq1} we have 
$$ \lim_{\varepsilon \to 0+} \exp \left( \int_{\xi}^\infty \frac{((\widehat{u}^\varepsilon(s))^k)'}{(\widehat{u}^\varepsilon(s))^k-s} ds \right)=1 $$ and
$$ \lim_{\varepsilon \to 0+} \widehat{v}^\varepsilon(\xi)= \lim_{\varepsilon \to 0+} v_+ \exp \left( \int_{\xi}^\infty \frac{((\widehat{u}^\varepsilon(s))^k)'}{(\widehat{u}^\varepsilon(s))^k-s} ds \right)=v_+, \quad \mbox{uniformly for } \xi>\xi_\sigma+\eta. $$
Similarly, we obtain also $\lim\limits_{\varepsilon\to0}\widehat{v}^\varepsilon(\xi)=v_{-}$, uniformly for $\xi<\xi_\sigma-\eta$.
\end{proof}
\begin{lemma}
Let $(\widehat{u}^\varepsilon, \widehat{v}^\varepsilon)$ be the solution of the Riemann problem \eqref{sysVis2}-\eqref{datsysVisc2}
 Denote
\begin{equation*} \label{limSigma}
 \sigma=\xi_\sigma=\lim_{\varepsilon \to 0+}
\xi_\sigma^\varepsilon=\lim_{\varepsilon \to 0+}
(\widehat{u}^\varepsilon(\xi_\sigma^\varepsilon))^k=(\widehat{u}(\sigma))^k.
\end{equation*}
  Then
  \begin{equation*}
   \lim_{\varepsilon \to 0+} (\widehat{v}^\varepsilon(\xi), \widehat{u}^\varepsilon(\xi))=\begin{cases}
                                                                          (v_-,u_-), &\mbox{if }\xi <\sigma,\\
                                                                          (w_0\cdot \delta, \sigma), &\mbox{if }\xi =\sigma,\\
                                                                          (v_+,u_+), &\mbox{if }\xi >\sigma,
                                                                         \end{cases}
  \end{equation*}
where $\widehat{v}^\varepsilon(\xi)$ converges in the sense of the distributions to the sum of a step function and a Dirac measure $\delta$ with weight $w_0=-\sigma(v_--v_+)+(v_-u_-^k-v_+u_+^k)$. Moreover, $\sigma = \frac{1}{k+1}\sum\limits_{j=0}^k u_-^{k-j}u_+^j$.
 \end{lemma}
 \begin{proof}
 From Lemma \ref{lemmaA} we have that $\sigma=\xi_\sigma=\frac{1}{k+1} \sum\limits_{j=0}^k u_-^{k-j}u_+^j$ and 
 $-\sigma (u_--u_+)+\frac{1}{k+1}(u_-^{k+1}-u_+^{k+1})=0.$
 Moreover, observe that $\psi_1(\theta)=\theta^k-\frac{1}{k+1}\frac{\theta^{k+1}-u_+^k}{\theta-u_+}>0$ for all $\theta > u_+$ and $\psi_2(\theta)=\frac{1}{k+1}\frac{u_-^k-\theta^{k+1}}{u_--\theta}-\theta^k>0$ for all $\theta < u_-$.
Then, as $\varepsilon \to 0+$, we have $u_+^k<\sigma <u_-^k$.
Now, we need to study the limit behavior of $\widehat{v}^\varepsilon$ in the neighborhood of $\sigma$. 
Let $\xi_1$ and $\xi_2$ be real numbers such that $\xi_1<\sigma<\xi_2$ and $\phi\in C_0^\infty([\xi_1,\xi_2])$ such that
$\phi(\xi)\equiv \phi(\sigma)$ for $\xi$ in a
neighborhood $\Omega$ of $\sigma$, $\Omega \subset (\xi_1,\xi_2)$ \footnote{The function $\phi$ is called a {\em sloping test function} \cite{TZZ}}. Then $\xi_\sigma^\varepsilon \in \Omega$ whenever $0<\varepsilon<\varepsilon_0$. 
From \eqref{sysVis2} we have
\begin{equation} \label{eqL1}
-\int_{\xi_1}^{\xi_2} \widehat{v}^\varepsilon((\widehat{u}^\varepsilon)^k-\xi)\phi' d\xi+\int_{\xi_1}^{\xi_2} \widehat{v}^\varepsilon \phi d\xi=0.
\end{equation}
For $\alpha_1, \alpha_2 \in \Omega$, $\alpha_1,\alpha_2$ near $\sigma$ such that $\alpha_1 <\sigma <\alpha_2$, we write
\begin{equation*}
 \int_{\xi_1}^{\xi_2} \widehat{v}^\varepsilon((\widehat{u}^\varepsilon)^k-\xi)\phi' d\xi=\int_{\xi_1}^{\alpha_1} \widehat{v}^\varepsilon((\widehat{u}^\varepsilon)^k-\xi)\phi' d\xi+\int_{\alpha_2}^{\xi_2} \widehat{v}^\varepsilon((\widehat{u}^\varepsilon)^k-\xi)\phi' d\xi,
\end{equation*}
and from Lemmas~\ref{lemmaA} and \ref{lemmaB}, we obtain
\begin{align*}
 \lim_{\varepsilon \to 0+} \int_{\xi_1}^{\xi_2} \widehat{v}^\varepsilon((\widehat{u}^\varepsilon)^k-\xi)\phi' d\xi&=
 \int_{\xi_1}^{\alpha_1} v_-(u_-^k-\xi)\phi' d\xi+\int_{\alpha_2}^{\xi_2} \widehat{v}_+(u_+^k-\xi)\phi' d\xi\\
 &= \left( v_-u_-^k-v_+u_+^k-v_-\alpha_1+v_+\alpha_2 \right) \phi(\sigma)\\
 &+\int_{\xi_1}^{\alpha_1}v_-\phi(\xi)d\xi
 +\int_{\alpha_2}^{\xi_2}v_+\phi(\xi)d\xi
\end{align*}
Then taking $\alpha_1 \to \sigma-$, $\alpha_2 \to \sigma+$, we arrive at
\begin{equation} \label{eqL2}
 \lim_{\varepsilon \to 0+} \int_{\xi_1}^{\xi_2} \widehat{v}^\varepsilon((\widehat{u}^\varepsilon)^k-\xi)\phi' d\xi=
 \left( -[\widehat{v}]\sigma+[\widehat{v} \widehat{u}^k] \right) \phi(\sigma)+\int_{\xi_1}^{\xi_2} J(\xi-\sigma)\phi(\xi)d\xi
\end{equation}
where $[q]=q_--q_+$ and $$J(x)=\begin{cases}
                               v_-, &\mbox{if }x<0,\\
                               v_+, &\mbox{if }x>0.
                              \end{cases}$$
From \eqref{eqL1} and \eqref{eqL2}, we get
\begin{equation*}
 \lim_{\varepsilon \to 0+} \int_{\xi_1}^{\xi_2} (\widehat{v}^\varepsilon-J(\xi-\sigma))\phi(\xi)d\xi=\left( -[\widehat{v}]\sigma+[\widehat{v} \widehat{u}^k] \right) \phi(\sigma).
\end{equation*}
for all sloping test functions $\phi \in C_0^\infty ([\xi_1,\xi_2])$.\\
For an arbitrary $\psi \in C_0^\infty([\xi_1,\xi_2])$, we take a sloping test function $\phi$, such that $\phi(\sigma)=\psi(\sigma)$ and $$ \max_{[\xi_1,\xi_2]}|\psi-\phi|<\mu, $$ for a sufficiently small $\mu>0$. As $\widehat{v}^\varepsilon \in L^1([\xi_1,\xi_2)$ uniformly, we obtain
\begin{align*}
 \lim_{\varepsilon \to 0+} \int_{\xi_1}^{\xi_2} (\widehat{v}^\varepsilon-J(\xi-\sigma))\psi(\xi)d\xi &= 
 \lim_{\varepsilon \to 0+} \int_{\xi_1}^{\xi_2} (\widehat{v}^\varepsilon-J(\xi-\sigma))\phi(\xi)d\xi+O(\mu)\\
 &=\left( -[\widehat{v}]\sigma+[\widehat{v} \widehat{u}^k] \right) \phi(\sigma)+O(\mu)\\
 &=\left( -[\widehat{v}]\sigma+[\widehat{v} \widehat{u}^k] \right) \psi(\sigma)+O(\mu).
\end{align*}
Then, when $\mu \to 0+$, we find that 
\begin{equation*} \label{deltaStr}
 \lim_{\varepsilon \to 0+} \int_{\xi_1}^{\xi_2} (\widehat{v}^\varepsilon-J(\xi-\sigma))\psi(\xi)d\xi=\left( -[\widehat{v}]\sigma+[\widehat{v} \widehat{u}^k] \right) \psi(\sigma)
\end{equation*}
holds for all test functions $\psi \in C_0^\infty ([\xi_1,\xi_2])$. Thus, $\widehat{v}^\varepsilon$ converges in the sense of the distributions to the sum of a step function and a Dirac delta function with strength $-[\widehat{v}]\sigma+[\widehat{v} \widehat{u}^k]$.
In similar way, we can show that
$$ \lim_{\varepsilon \to 0+} \int_{\xi_1}^{\xi_2} (\widehat{u}^\varepsilon - \widetilde{J}(\xi -\sigma)) \psi(\xi) d\xi =0 $$
for all test functions $\psi \in C_0^\infty ([\xi_1,\xi_2])$ and where 
$\widetilde{J}(x) =
\begin{cases}
u_-, &\mbox{if } x<0, \\
u_+, &\mbox{if } x>0.
\end{cases}$. \\
Thus,
 $\widehat{u}^\varepsilon$ converges in the sense of the distributions to a step function.
\end{proof}
Then we get the following theorem.
\begin{theorem} \label{ThmFinal}
 Suppose $u_->u_+$. Let $(\widehat{v}^\varepsilon(x,t),\widehat{u}^\varepsilon(x,t))$ be the similarity solution of \eqref{sysVis1}--\eqref{datsysVisc1}. Then the limit $$ \lim_{\varepsilon \to 0+}(\widehat{v}^\varepsilon(x,t),\widehat{u}^\varepsilon(x,t))=(\widehat{v}(x,t),\widehat{u}(x,t)) $$
 exists in the measure sense and $(\widehat{v},\widehat{u})$ solves \eqref{sys1}--\eqref{datsysVisc1}.
 Moreover, 
 \begin{equation*}
  (\widehat{v}(x,t),\widehat{u}(x,t))=\begin{cases}
                      (v_-,u_-), &\mbox{if } x<\frac{\sigma}{\alpha k} (1-e^{-\alpha kt}),\\
                      (\frac{w_0}{\alpha k}(1-e^{-\alpha kt})\delta(x-\frac{\sigma}{\alpha k} (1-e^{-\alpha kt})), \sigma), &\mbox{if } x=\frac{\sigma}{\alpha k} (1-e^{-\alpha kt}),\\
                      (v_+,u_+), &\mbox{if } x>\frac{\sigma}{\alpha k} (1-e^{-\alpha kt}),
                     \end{cases}
 \end{equation*}
where $\sigma = \frac{1}{k+1}\sum\limits_{j=0}^k u_-^{k-j}u_+^j$ and $w_0=-\sigma(v_--v_+)+(v_-u_-^k-v_+u_+^k)$. Moreover, $\sigma$ satisfies the entropy condition $u_+^k <\sigma <u_-^k$.
\end{theorem}

{\bf Case 2. $u_-<u_+$} 
\begin{lemma}
For any $\eta >0$,
$$ \lim_{\varepsilon \to 0+} \widehat{u}_\xi^\varepsilon(\xi) =0, \mbox{ for } \xi \le u_-^k-\eta \mbox{ or } \xi \ge u_+^k+\eta, $$
$$
\lim_{\varepsilon \to 0+} (\widehat{v}^\varepsilon(\xi), \widehat{u}^\varepsilon(\xi))=
\begin{cases}
(v_-,u_-), &\mbox{if } \xi < \xi_{\sigma_1}-\eta, \\
(0,\xi), &\mbox{if }  \xi_{\sigma_1}-\eta \le \xi \le \xi_{\sigma_2}+\eta, \\
(v_+,u_+), &\mbox{if } \xi > \xi_{\sigma_2}+\eta,
\end{cases}
$$
uniformly in the above intervals.
\end{lemma}
\begin{proof}
Since $\widehat{u}$ is a increasing smooth function in $(-\infty,+\infty)$, then 
$u_- \le \widehat{u}(\xi_{\sigma_1}^\varepsilon) \le \widehat{u}(\xi) \le \widehat{u}(\xi_{\sigma_2}^\varepsilon) \le u_+$ or $u_-^k \le \xi_{\sigma_1}^\varepsilon \le (\widehat{u}(\xi))^k \le \xi_{\sigma_2}^\varepsilon \le u_+^k$.

The proof of this lemma is basically similar to that of Lemma \ref{lemmaA}.
Take $\xi_3 = \xi_{\sigma_1} -\eta/2$ and let $\varepsilon$ be so small such that $\xi_{\sigma_1}^\varepsilon > \xi_3 + \eta/4$. Integrating the  first equation of \eqref{sysVis2} twice on $[\xi,\xi_3]$, we get
\begin{align*}
\widehat{u}(\xi_3)-\widehat{u}(\xi)&=\widehat{u}'(\xi_3) \int_{\xi}^{\xi_3} \exp \left( -\int_r^{\xi_3} \frac{(\widehat{u}(s))^k-s}{\varepsilon}ds \right) dr\\
& \ge \widehat{u}'(\xi_3) \int_{\xi}^{\xi_3} \exp \left( -\int_r^{\xi_3} \frac{u_-^k-s}{\varepsilon}ds \right) dr\\
&= \widehat{u}'(\xi_3) \int_\xi^{\xi_3} \exp \left( \frac{1}{\varepsilon} \left( \left(u_-^k-\xi_3 \right)(r-\xi_3)-\frac12(r-\xi_3)^2 \right) \right) dr\\
  &= \widehat{u}'(\xi_3) \int_{\xi-\xi_3}^0 \exp \left( \frac{1}{\varepsilon} \left( \left(u_-^k-\xi_3 \right)r-\frac12 r^2 \right) \right) dr.
\end{align*}
Letting $\xi \to -\infty$, we get
\begin{align*}
 u_+-u_- &\ge \widehat{u}'(\xi_3) \int_{-\infty}^0 \exp \left( \frac{1}{\varepsilon} \left( \left(u_-^k-\xi_3 \right)r-\frac12 r^2 \right) \right) dr\\
 &\ge \widehat{u}'(\xi_3) \int_{0}^{2\varepsilon} \exp \left( -\frac{1}{\varepsilon} \left( \left(u_-^k-\xi_3 \right)r+\frac12 r^2 \right) \right) dr\\
 &\ge \widehat{u}'(\xi_3) \sqrt{\varepsilon}A_3
\end{align*}
for $0\le \varepsilon \le 1$, where $A_3$ is a constant independent of $\varepsilon$. Thus
\begin{equation*}
 |\widehat{u}'(\xi_3)| \le \frac{u_+-u_-}{\sqrt{\varepsilon}A_3}.
\end{equation*}
Noticing that 
\begin{align*}
(\widehat{u}(s))^k-s &=((\widehat{u}(s))^k-(\widehat{u}(\xi_{\sigma_1}^\varepsilon))^k)-(s-\xi_{\sigma_1}^\varepsilon) 
= (k (\widehat{u}(\theta))^{k-1}u'(\theta)-1)(s-\xi_{\sigma_1}^\varepsilon) \ge \frac{\eta}{4}
\end{align*}
for $s \le \xi_3$ and from \eqref{eqC} we have
\begin{equation*}
 |\widehat{u}'(\xi)| \le \frac{u_+-u_-}{\sqrt{\varepsilon}A_3}\exp \left( -\frac{\eta}{4\varepsilon} \left( \xi_3-\xi \right) \right)
\end{equation*}
which implies that
$$ \lim_{\varepsilon \to 0+} \widehat{u}_\xi^\varepsilon(\xi)=0, \qquad \mbox{uniformly for }\xi \le \xi_{\sigma_1}-\eta. $$
Now, we choose $\xi$ and $\xi_4$ such that $\xi<\xi_4 \le \xi_{\sigma_1}-\eta$.
From
$$ \widehat{u}(\xi_4)-\widehat{u}(\xi)=\widehat{u}'(\xi_4)\int_\xi^{\xi_4} \exp \left(- \int_r^{\xi_4}
\frac{(\widehat{u}(s))^k-s}{\varepsilon} ds \right) dr, $$
we get
\begin{align*}
 | \widehat{u}(\xi_4)-\widehat{u}(\xi)| &\le |\widehat{u}'(\xi_4)| \int_\xi^{\xi_4} \exp \left(
\frac{A_4}{\varepsilon}(r-\xi_4) \right)dr \le \frac{\varepsilon}{A_4}|\widehat{u}'(\xi_4)| \left(1-\exp \left(
\frac{A_4}{\varepsilon}(\xi-\xi_4) \right) \right),
\end{align*}
where $A_4=(\widehat{u}(\xi_4))^k - \xi_4$.
When $\xi \to-\infty$, we obtain $ |\widehat{u}(\xi_4)-u_-| \le \frac{\varepsilon}{A_4}|\widehat{u}'(\xi_4)|$,
which implies that
\begin{equation*}
 \lim_{\varepsilon \to 0+} \widehat{u}^\varepsilon(\xi)=u_-, \qquad \mbox{uniformly
for } \xi < \xi_{\sigma_1}-\eta.
\end{equation*}
The results for $\xi >\xi_{\sigma_2}+\eta$ can be obtained analogously.\\
Now, noticing that for $\xi < \xi_{\sigma_1}$,
\begin{align}
\widehat{v}_1(\xi) &= v_- \exp \left( - \int_{-\infty}^\xi \frac{((\widehat{u}(s))^k)'}{(\widehat{u}(s))^k-s} ds \right)
=\lim_{R \to +\infty} v_- \exp \left( - \int_{-R}^\xi \frac{((\widehat{u}(s))^k-s)'+1}{(\widehat{u}(s))^k-s} ds \right) \nonumber \\
& \ge \lim_{R \to +\infty} v_- \left( \frac{u_-^k+R}{(\widehat{u}(\xi))^k-\xi} \right) \exp \left( - \int_{-R}^\xi \frac{ds}{u_-^k-s} \right)
= v_- \left( \frac{u_-^k-\xi}{(\widehat{u}(\xi))^k-\xi} \right). \label{desg1}
\end{align}
By Lemma \ref{Lemma2.3}, $\widehat{v}_1(\xi)$ is decreasing for $\xi < \xi_{\sigma_1}$ and from \eqref{desg1} we have
$$ v_- \ge \widehat{v}_1(\xi) \ge v_- \left( \frac{u_-^k-\xi}{(\widehat{u}(\xi))^k-\xi} \right) \to v_-, \quad \mbox{as } \varepsilon \to 0+. $$
Thus, $$ \lim\limits_{\varepsilon \to 0+} \widehat{v}^\varepsilon(\xi)=v_-, \quad
\mbox{uniformly for } \xi < \xi_{\sigma_1}-\eta. $$
Analogously,  we obtain $\lim\limits_{\varepsilon \to 0+} \widehat{v}^\varepsilon(\xi)=v_+$, uniformly for $\xi > \xi_{\sigma_2}+\eta$.
From Lemma \ref{Lemma2.3}, on $[\xi_{\sigma_1},\xi_{\sigma_2}]$ 
we have that $\widehat{v}(\xi)=0$.
Now, choose $\eta_1 >0$ and let $\phi \in C_0^\infty((\xi_1,\xi_2))$ where $\xi_1 < \xi_{\sigma_1}-\eta_1 < \xi_2$. From \eqref{Eqv2} we have
$$ 0=\int_{\xi_1}^{\xi_2} (\widehat{v}(\xi) (\xi \phi(\xi))'-\widehat{v}(\xi) (\widehat{u}(\xi))^k \phi'(\xi)) d\xi
= \int_{\xi_1}^{\xi_{\sigma_1-\eta_1}} (v_- (\xi \phi(\xi))'-v_- u_-^k \phi'(\xi)) d\xi. $$
Thus, we have
$ v_- (\xi_{\sigma_1-\eta_1} \phi(\xi_{\sigma_1-\eta_1})- u_-^k \phi(\xi_{\sigma_1-\eta_1}))=0 $
which yields $\xi_{\sigma_1}= u_-^k$ for arbitrary $\phi$ and arbitrary $\eta_1$.
Analogously,  we obtain $\xi_{\sigma_2}= u_+^k$.\\
For $\xi \in [\xi_{\sigma_1} + \eta, \xi_{\sigma_2} - \eta]$, denote $\lim\limits_{\varepsilon \to 0+} \widehat{u}^\varepsilon(\xi)=\widehat{u}(\xi)$. Thus, from the chain rule of Volpert for BV functions \cite{Volpert, LeFloch}, Eq. \eqref{ViscEq1} and \eqref{Eqv2}, we have that $(\widehat{u}(\xi))^k=\xi$ with $\widehat{u}(u_-^k)=u_-$ and $\widehat{u}(u_+^k)=u_+$. Also, (with Lemma \ref{Lemma2.3}) we have 
$\lim\limits_{\varepsilon \to 0+} \widehat{v}^\varepsilon(\xi)=0$.
\end{proof}

Now, we study the limit behavior of $(\widehat{v}^\varepsilon, \widehat{u}^\varepsilon)$ 
as $\varepsilon \to 0+$.
\begin{theorem}
Suppose $u_+>u_-$. Let $(\widehat{u}^\varepsilon, \widehat{v}^\varepsilon)$ be the solution of the Riemann problem \eqref{sysVis2}-\eqref{datsysVisc2}. Then, $\lim\limits_{\varepsilon \to 0+} (\widehat{v}^\varepsilon(x,t),\widehat{u}^\varepsilon(x,t))
= (\widehat{v}(x,t),\widehat{u}(x,t))$ exists in the sense of distributions and $(\widehat{v},\widehat{u})$ solves \eqref{sys1}-\eqref{datsysVisc1}. Moreover,
$$
(\widehat{v}(x,t),\widehat{u}(x,t)) = 
\begin{cases}
(v_-,u_-), &\mbox{if } x< \frac{u_-^k}{\alpha k} (1-e^{-\alpha kt}), \\
(0,(\frac{\alpha kx}{1-e^{-\alpha kt}})^{1/k}), &\mbox{if } \frac{u_-^k}{\alpha k} (1-e^{-\alpha kt}) \le x \le \frac{u_+^k}{\alpha k} (1-e^{-\alpha kt}) ,\\
(v_+,u_+), &\mbox{if } x> \frac{u_+^k}{\alpha k} (1-e^{-\alpha kt}).
\end{cases}
$$
\end{theorem}

\section{Riemann problem for the system \texorpdfstring{\eqref{system_ld}}{(1)}}
In this section, we study the Riemann problem to the original system \eqref{system_ld}. When $u_-<u_+$, the solution of \eqref{system_ld}--\eqref{datoRiemann} 
is directly obtained from the corresponding ones to \eqref{sys1}--\eqref{datsysVisc1} by performing the transformation of state variables $(v(x,t),u(x,t)) = (\widehat{v}(x,t),\widehat{u}(x,t)e^{-\alpha t})$, in which the positions of the contact discontinuities remain unchanged. Then, we have the following result for classical Riemann solutions.
\begin{theorem} \label{ThmCRS}
Assume that $u_-<u_+$. Then the solution for the Riemann problem is
$$
(v(x,t),u(x,t)) = 
\begin{cases}
(v_-,u_-e^{-\alpha t}), &\mbox{if } x< \frac{u_-^k}{\alpha k} (1-e^{-\alpha kt}), \\
(0,(\frac{\alpha kx}{1-e^{-\alpha kt}})^{1/k}e^{-\alpha t}), &\mbox{if } \frac{u_-^k}{\alpha k} (1-e^{-\alpha kt}) \le x \le \frac{u_+^k}{\alpha k} (1-e^{-\alpha kt}) ,\\
(v_+,u_+e^{-\alpha t}), &\mbox{if } x> \frac{u_+^k}{\alpha k} (1-e^{-\alpha kt}).
\end{cases}
$$
\end{theorem}
It is clear that the above theorem generalizes the Theorem 3.1 in \cite{Delacruz}.
Now, we study the case when $u_- > u_+$. We need recall the following definition:
\begin{definition}
 A two-dimensional weighted delta function $w(s)\delta_L$ supported on a smooth curve $L=\{ (x(s),t(s))\, :\, a < s < b \}$, for $w\in L^1((a,b))$, is defined as
 \begin{equation*}
  \langle w(\cdot)\delta_L,\phi(\cdot,\cdot) \rangle = \int_a^b w(s) \phi(x(s),t(s)) \, ds, \quad \phi \in C_0^\infty(\re \times[0,\infty)). 
 \end{equation*}
\end{definition}
Now, we define a delta shock wave solution for the system \eqref{system_ld} with initial data \eqref{datoRiemann}.
\begin{definition}
 A distribution pair  $(v,u)$ is a {\em delta shock wave solution} of \eqref{system_ld} and \eqref{datoRiemann} in the sense of distribution if there exist a smooth curve $L$ and a function $w \in C^1(L)$ such that $v$ and $u$ are represented in the following form
 $$ v = \widetilde{v}(x,t)+w \delta_L \text{ and } u=\widetilde{u}(x,t), $$
 $\widetilde{v}, \widetilde{u} \in L^\infty(\re \times(0,\infty); \re)$ and
 \begin{equation} \label{weakSol1}
  \begin{cases}
  \langle u, \varphi_t \rangle + \langle u^{k+1} , \varphi_x \rangle = \int_0^\infty \int_\mathbb{R} \alpha u \varphi dx dt,\\
   \langle v, \varphi_t \rangle + \langle v u^k, \varphi_x \rangle =0,
  \end{cases}
 \end{equation}
for all the test functions $\varphi \in C_0^\infty (\re \times(0,\infty))$, where $u|_L=u_\delta(t)$ and
\begin{align*}
 \langle v,\varphi \rangle &= \int_0^\infty \int_\re \widetilde{v} \varphi \, dx dt + \langle w \delta_L, \varphi \rangle,\\
 \langle v G(u),\varphi \rangle &= \int_0^\infty \int_\re \widetilde{v} G(\widetilde{u}) \varphi \, dx dt + \langle w G(u_\delta) \delta_L, \varphi \rangle.
\end{align*}
\end{definition}
With the previous definitions, we are going to find a solution with discontinuity $x=x(t)$ for \eqref{system_ld} of the form
\begin{equation} \label{deltaSol1}
 (v(x,t),u(x,t))=\begin{cases}
                          (v_-(x,t),u_-(x,t)), &\text{if } x<x(t),\\
                          (w(t)\delta_L,u_\delta(t)), &\text{if } x=x(t),\\
                          (v_+(x,t),u_+(x,t)), &\text{if } x>x(t),
                      \end{cases}
\end{equation}
where $v_\pm(x,t)$, $u_\pm(x,t)$ are piecewise smooth solutions of system \eqref{system_ld}, $\delta(\cdot)$ is the 
 Dirac measure supported on the curve $x(t) \in C^1$, and $x(t)$, $w(t)$ and $u_\delta(t)$ are to be determined.\\
 
Since $v(x,t)=\widehat{v}(x,t)$ and $u(x,t)=\widehat{u}(x,t)e^{-\alpha t}$, from Theorem \ref{ThmFinal}, we can establish a solution of the form \eqref{deltaSol1} to the system \eqref{system_ld} with initial data \eqref{datoRiemann}. Thus, we have the following result.
\begin{theorem} \label{Thm4.1}
 Assume that $u_->u_+$.
 Then the Riemann problem \eqref{system_ld}--\eqref{datoRiemann} admits one and only one measure solution of the
form
\begin{equation} \label{sol_final}
  (v(x,t),u(x,t))=\begin{cases}
                      (v_-,u_-e^{-\alpha t}), &\mbox{if } x<x(t),\\
                      (w(t) \delta(x-x(t)), \sigma e^{-\alpha t}), &\mbox{if } x=x(t),\\
                      (v_+,u_+ e^{-\alpha t}), &\mbox{if } x>x(t),
                     \end{cases}
 \end{equation}
where $w(t)=\frac{w_0}{\alpha k}(1-e^{-\alpha kt})$, $x(t)=\frac{\sigma}{\alpha k} (1-e^{-\alpha kt})$, $\sigma=\sum\limits_{j=0}^k u_-^{k-j}u_+^j$ and $w_0=-\sigma(v_--v_+)+(v_- u_-^k-v_+ u_+^k)$. Moreover, $dx(t)/dt$ satisfies the entropy condition $u_+^k e^{-\alpha t} <dx(t)/dt < u_-^k e^{-\alpha t}$ for all $t \ge 0$. 
\end{theorem}
\begin{proof}
We need show that \eqref{sol_final} is a solution to the problem \eqref{system_ld}--\eqref{datoRiemann} which can be found with $(v,u)=(\widehat{v},\widehat{u}e^{-\alpha t})$ and the result obtained in Theorem \ref{ThmFinal}. 
Therefore, for any test function $\varphi \in C_0^\infty (\re \times (0,\infty))$ we have
\begin{align*}
\langle u, \varphi_t \rangle + \langle u^{k+1} , \varphi_x \rangle =& \int_0^\infty \int_\re (u \varphi_t+u^{k+1} \varphi_x) dx dt \\
=&\int_0^\infty \int_{-\infty}^{x(t)} (u_-e^{-\alpha t} \varphi_t+u_-^{k+1}e^{-\alpha (k+1)t} \varphi_x) dx dt \\
&+ \int_0^\infty \int_{x(t)}^{\infty} (u_+e^{-\alpha t} \varphi_t+u_+^{k+1}e^{-\alpha (k+1)t} \varphi_x) dx dt\\
=&-\oint - \left( u_-^{k+1}e^{-\alpha (k+1)t} \varphi \right) dt + \left( u_-e^{-\alpha t} \varphi \right) dx\\
&+ \oint - \left( u_+^{k+1}e^{-\alpha (k+1)t} \varphi \right) dt + \left( u_+e^{-\alpha t} \varphi \right) dx
+ \int_0^\infty \int_{\re} \alpha u \varphi dx dt\\
=&\int_0^\infty \left( (u_-^{k+1}-u_+^{k+1})e^{-\alpha kt} -\frac{dx(t)}{dt} (u_--u_+)  \right) e^{-\alpha t} \varphi dt
+ \int_0^\infty \int_{\re} \alpha u \varphi dx dt\\
=& \int_0^\infty \int_\mathbb{R} \alpha u \varphi dx dt
\end{align*}
which implies the second equation of \eqref{weakSol1}. A completely similar argument leads to the first equation of \eqref{weakSol1}.
\begin{align*}
\langle v, \varphi_t \rangle + \langle vu^{k} , \varphi_x \rangle =& \int_0^\infty \int_\re (v \varphi_t+vu^{k} \varphi_x) dx dt 
+\int_0^\infty w(\varphi_t + u_\delta^k \varphi_x) dt \\
=&\int_0^\infty \int_{-\infty}^{x(t)} (v_- \varphi_t+v_-u_-^{k}e^{-\alpha kt} \varphi_x) dx dt \\
&+ \int_0^\infty \int_{x(t)}^{\infty} (v_+ \varphi_t+v_+u_+^{k}e^{-\alpha kt} \varphi_x) dx dt 
+\int_0^\infty w(\varphi_t + u_\delta^k \varphi_x) dt \\
=&-\oint - \left( v_-u_-^{k}e^{-\alpha kt} \varphi \right) dt + \left( v_- \varphi \right) dx
+ \oint - \left( v_+u_+^{k}e^{-\alpha kt} \varphi \right) dt + \left( v_+ \varphi \right) dx\\
&+ \int_0^\infty w \frac{d\varphi}{dt} dt\\
=&\int_0^\infty \left( (v_-u_-^{k}-v_+u_+^{k})e^{-\alpha kt} -(v_--v_+)\frac{dx(t)}{dt}-\frac{dw(t)}{dt} \right) \varphi dt=0
\end{align*}
\end{proof}

\section{Final remarks} 
From Theorem \ref{ThmCRS}, we can observe that when $\alpha \to 0+$, the solution converges to
$$
(v(x,t),u(x,t))=
\begin{cases}
(v_-,u_-), &\mbox{if } x <u_-^k t,\\
(0, (x/t)^{1/k}), &\mbox{if } u_-^k t \le x \le u_+^k t,\\
(v_+,u_+), &\mbox{if } x >u_+^k t,
\end{cases}
$$
which is the classical Riemann solution for the homogeneous system associated to \eqref{system_ld}.
In similar way, from Theorem \ref{Thm4.1}, we can observe that when $\alpha \to 0+$, the solution converges to
$$
(v(x,t),u(x,t))=
\begin{cases}
(v_-,u_-), &\mbox{if } x <\sigma t,\\
(w_0 t \delta(x-\sigma t), \sigma), &\mbox{if } x = \sigma t,\\
(v_+,u_+), &\mbox{if } x >\sigma t,
\end{cases}
$$
where $\sigma = \frac{1}{k+1}\sum\limits_{j=0}^{k} u_-^{k-j}u_+^j$ and $w_0=-\sigma(v_--v_+)+(v_-u_-^k-v_+u_+^k)$. This solution is a delta shock wave solution for the homogeneous system associated to \eqref{system_ld}. 
The Riemann problem for the homogeneous system associated to \eqref{system_ld} with $k=1$
was solved by K.T. Joseph (see main theorem in \cite{Joseph}).\\

%

{\bf Acknowledgments}\\
The first author wishes to thank Professor Kayyunnapara Thomas Joseph who kindly sent him the paper \cite{Joseph}. 

\end{document}